\documentclass[12pt]{amsart}
\usepackage{amssymb}
\title{Almost commuting varieties for the symplectic Lie algebras}
\author{Ivan Losev}
\renewcommand{\sp}{\mathfrak{sp}}
\renewcommand{\sl}{\mathfrak{sl}}

\newcommand{\C}{\mathbb{C}}
\newcommand{\Sp}{\operatorname{Sp}}
\newcommand{\GL}{\operatorname{GL}}
\newcommand{\g}{\mathfrak{g}}
\newcommand{\quo}{/\!/}
\newcommand{\h}{\mathfrak{h}}
\newcommand{\rk}{\operatorname{rk}}
\newtheorem{Thm}{Theorem}[section]

\newtheorem{Cor}[Thm]{Corollary}
\newtheorem{Lem}[Thm]{Lemma}
\theoremstyle{definition}

\newtheorem{Rem}[Thm]{Remark}

\numberwithin{equation}{section}
\begin{document}
\begin{abstract}
In this note we introduce and study the almost commuting varieties for the symplectic Lie algebras.
\end{abstract}
\maketitle

\section{Introduction}\label{S_intro}
The commuting schemes of  semisimple Lie  algebras is a classical object of study in
Lie theory and Invariant theory. Let $\g$ be a semisimple Lie algebra.
Consider the subscheme $C_2(\g)\subset \g^{\oplus 2}$ defined by the equation $[x,y]=0$.
This is the commuting scheme for $\g$. Note that we can define $C_k(\g)\subset \g^{\oplus k}$
analogously but in this note we restrict ourselves to the case of $k=2$.

The algebro-geometric properties of $C_2(\g)$ are largely unknown. It is known that this
scheme is irreducible, \cite{Richardson}. But it is not known whether $C_2(\g)$
is reduced (or whether $C_2(\g)$ with reduced scheme structure is normal) even in
the case of $\g=\mathfrak{sl}_n$. However, for $\g=\sl_n$ there is a related
variety, the {\it almost commuting} variety. The version we need was introduced
and studied in detail in \cite{GG}.

We now recall the definition. Consider the vector space $R:=\sl_n^{\oplus 2}\oplus \C^n\oplus
(\C^{n})^*$. We write $(x,y,i,j)$ for a typical element of $R$. We have the subscheme $M_n\subset R$ defined by $[x,y]+ij=0$. The following is the main result of \cite{GG}.

\begin{Thm}[Theorem 1.1.2 in \cite{GG}]
The scheme $M_n$ is a reduced complete intersection.
It has $n+1$ irreducible components, all of which have dimension $n^2+2n-2$.
\end{Thm}

The goal of this note is to define the analog of $M_n$ for the symplectic Lie algebra
$\g=\sp_{2n}$, study its properties and apply to the study of the categorical quotient
$C_2(\g)\quo G$ (where $G=\operatorname{Sp}_{2n}$).

We will need a notation. Let $\C^{2n}$ be the tautological representation of $\sp_{2n}$.
We can identify $S^2(\C^{2n})$ with  $\sp_{2n}$ in the standard way (the Lie bracket on
$S^2(\C^{2n})$ is the restriction of the Poisson bracket). So for $i\in \C^{2n}$ we can view $i^2\in S^2(\C^{2n})$
as an element of $\sp_{2n}$.
The almost commuting scheme for $\sp_{2n}$ is defined by
\begin{equation}\label{eq:almost_commut_C}
X_n=\{(x,y,i)\in \sp_{2n}^{\oplus 2}\oplus \C^{2n}| [x,y]+i^2=0\}.
\end{equation}

Here is our main result concerning $X_n$. Somewhat surprisingly, its algebro-geometric properties
are even better than those of $M_n$:

\begin{Thm}\label{Thm:X_n_properties}
The scheme $X_n$ is a reduced complete intersection of dimension $2n^2+3n$.
It is irreducible.
\end{Thm}

Here is an application of $X_n$ to the study of the commuting scheme $C_2(\g)$ (with $\g=\mathfrak{sp}_{2n}$). Note that $G$ naturally acts on $X_n$. Let $\h,W$
be the Cartan subalgebra and the Weyl group of $\g$.

\begin{Thm}\label{Thm:X_n_quotient}
We have scheme isomorphisms $X_n\quo G\xrightarrow{\sim} C_2(\g)\quo G\xrightarrow{\sim}
\h^{\oplus 2}/W$, where $W$ acts on $\h^{\oplus 2}$ diagonally.
\end{Thm}

Note that the isomorphism $C_2(\g)\quo G\xrightarrow{\sim}\h^{\oplus 2}/W$ is a special
case (for $k=2$) of the main result of \cite{CN}.

{\bf Acknowledgements}. I would like to thank Tsao-Hsien Chen for his talk on \cite{CN} at Yale
which motivated this work. My work was partially supported by the NSF
under grant DMS-2001139.

\section{Properties of almost commuting variety}
\subsection{Upper triangularity property}
Here we prove the following easy result.

\begin{Lem}\label{Lem:upper_triangularity}
Let $(x,y,i)\in X_n$. Then there is a Borel subalgebra of $\mathfrak{sp}_{2n}$
containing both $x$ and $y$.
\end{Lem}
\begin{proof}
We note that $\rk[x,y]=\rk (i^2)\leqslant 1$. So $x,y$ have a common eigenvector, see, e.g.,
\cite[Lemma 12.7]{EG}. Denote this vector by $v$ and let $v^\perp$ denote its skew-orthogonal
complement. The subspace $v^\perp$ is $x$- and $y$-stable. Let $x_1,y_1$ be the operators
on $v^\perp/\C v$ induced by $x,y$. Then we have $\rk [x_1,y_1]\leqslant \rk[x,y]\leqslant 1$.
So we can argue by induction to show that $x,y$ preserve a lagrangian flag. Equivalently,
they are contained in some Borel subalgebra of $\g$.
\end{proof}

Here is a standard corollary.

\begin{Cor}\label{Cor:closed_orbits}
Let $(x,y,i)\in X_n$ be such that $G(x,y,i)$ is closed. Then $[x,y]=0, i=0$ and $x,y$
are semisimple. Conversely, the orbit of such a triple is closed.
\end{Cor}

\begin{Rem}\label{Rem:upper_triangularity}
One can ask for which simple Lie algebras $\g$ the claim Lemma \ref{Lem:upper_triangularity} holds,
where we consider the condition $[x,y]\in \mathbb{O}_{min}$ for the minimal
nilpotent orbit $\mathbb{O}_{min}$. By \cite[Lemma 12.7]{EG} it holds for $\g=\mathfrak{sl}_n$. And it also holds for $\g=\mathfrak{sp}_{2n}$. In fact, it does not hold for any other simple
Lie algebra. Indeed, $\dim \mathbb{O}_{min}>2\dim \h$ for all $\g$ different from
$\mathfrak{sl}_n,\mathfrak{sp}_{2n}$. Assuming the inequality holds, consider a regular element
$x\in \mathfrak{h}$. For any element $z\in \mathbb{O}_{min}\cap \h^{\perp}$ we can find
$y\in \g$ with $[x,y]=z$. On the other hand, if $x,y$ lie in a Borel subalgebra $\mathfrak{b}$,
then $\mathfrak{b}$ is one of $|W|$ Borel subalgebras of $\g$ containing $x$. For any such
$\mathfrak{b}$, we have $\dim (\mathbb{O}_{min}\cap \mathfrak{b})=\frac{1}{2}\dim \mathbb{O}_{min}$,
see, e.g., \cite[Theorem 3.3.7]{CG}. If $\dim(\mathbb{O}_{min}\cap \h^\perp)> \dim (\mathbb{O}_{min}\cap \mathfrak{b})$, then we can find $y$ such that $x$ and $y$ do not lie in the same
Borel subalgebra but $[x,y]\in \mathbb{O}_{min}$.
 \end{Rem}

\subsection{Local structure}
The goal of this section is to describe the structure of $X_n$ near a closed $G$-orbit.
Recall that the closed orbits are described by Corollary \ref{Cor:closed_orbits}.

Let $p:=(x,y,0)$ be a point with closed $G$-orbit. Then $x,y$ are commuting semisimple elements.
The common centralizer, $L$, of $x$ and $y$ is a Levi subgroup in $G$, hence has the form
$\prod_{i=1}^k \GL_{n_i}\times \Sp_{2n_0}$ for a partition $n=n_0+n_1+\ldots+n_k$ into
the sum of positive integers. Consider the subscheme $\C^{2k}\times \prod_{i=1}^k M_{n_i}\times X_{n_0}
\subset \mathfrak{l}^{\oplus 2}\oplus \C^{2n}$,
where $M_{n_i}$ was defined in Section \ref{S_intro} and $\C^{2k}$ is identified
with $\mathfrak{z}(\mathfrak{l})^{\oplus 2}$. It comes with an action of $L$.
We can form the homogeneous bundle
$G\times^L (\C^{2k}\oplus \prod_{i=1}^k M_{n_i}\times X_{n_0})$. Let $X_n^{\wedge_{Gp}}$ denote the spectrum of the completion of $\C[X_n]$ at the ideal of all functions vanishing at $Gp$. Similarly,
we can consider the scheme $\left(G\times^L (\C^{2k}\times\prod_{i=1}^k M_{n_i}\times X_{n_0})\right)^{\wedge_{G/L}}$, here we complete with respect to the ideal of all functions on $G\times^L (\C^{2k}\times\prod_{i=1}^k M_{n_i}\times X_{n_0})$ vanishing at the orbit $G\times^L\{0\}$.

\begin{Lem}\label{Lem:local_structure}
We have a $G$-equivariant scheme isomorphism
$$X_n^{\wedge_{Gp}}\xrightarrow{\sim} \left(G\times^L (\C^{2k}\times\prod_{i=1}^k M_{n_i}\times X_{n_0})\right)^{\wedge_{G/L}}.$$
\end{Lem}
\begin{proof}
Note that the action of $G$ on $\g^{\oplus 2}\oplus \C^{2n}$ is Hamiltonian with moment
map $\mu: \g^{\oplus 2}\oplus \C^{2n}\rightarrow \g$ given by $\mu(x,y,i)=[x,y]+i^2$.
So we can apply the main result of \cite{slice} to this action and the point $p$.
Note that this result is stated in \cite{slice} for neighborhoods in the usual topology,
but it works for formal neighborhoods as well.
We get the isomorphism between the following two symplectic schemes with Hamiltonian
actions (that, in particular, intertwine the moment maps).
On the one side, we have $(\g^{\oplus 2}\oplus \C^{2n})^{\wedge_{Gp}}$.
On the other side, we have $\left(T^*(G\times^L \mathfrak{l})\times \C^{2n}\right)^{\wedge_{G/L}}$
with  natural symplectic form and moment map. The zero locus of the moment map in
$(\g^{\oplus 2}\oplus \C^{2n})^{\wedge_{Gp}}$ (as a scheme) is  $X_n^{\wedge_{Gp}}$.
The analogous locus in $\left(T^*(G\times^L \mathfrak{l})\times \C^{2n}\right)^{\wedge_{G/L}}$
is $\left(G\times^L (\C^{2k}\times\prod_{i=1}^k M_{n_i}\times X_{n_0})\right)^{\wedge_{G/L}}$.
This yields the required isomorphism.
\end{proof}

\begin{Rem}\label{Rem:non_normality}
Using Lemma \ref{Lem:local_structure} and the fact that $M_n$ is not irreducible, one sees that
$X_n$ is not normal.
\end{Rem}

\subsection{Dimension bound}
Here we are going to prove a technical lemma. Consider the map $\pi:X_n\rightarrow \g$ given by projection
to the first component. Let $\mathbb{O}\subset \g$ be an adjoint orbit.

\begin{Lem}\label{Lem:dim_preimage}
We have $\dim \pi^{-1}(\mathbb{O})<2n^2+3n$.
\end{Lem}
\begin{proof}
Fix $x\in \mathbb{O}$. We need to show that $X_{n,x}:=\{(y,i)| [x,y]=i^2\}$ has dimension
less then $\dim \mathfrak{z}_{\g}(x)+2n$. Also consider the varieties $$\underline{X}_{n,x}:=\{(y,z)| [x,y]=z, z\in \mathbb{O}_{min}\}, Y_{n,x}:=\mathbb{O}_{min}\cap [\g,x].$$
We have the natural maps $\rho_1:X_{n,x}\mapsto \underline{X}_{n,x}, (y,i)\mapsto (y,i^2)$
and $\rho_2:\underline{X}_{n,x}\rightarrow Y_{n,x}, (y,z)\mapsto z$. We note that $\rho_1$
is finite, while $\rho_2$ is an affine bundle with fiber of dimension $\dim \mathfrak{z}_{\g}(x)$.
So we reduce to proving that $\dim Y_{n,x}<2n=\dim \mathbb{O}_{min}$, equivalently, that
$Y_{n,x}\neq \mathbb{O}_{min}$, equivalently, $\mathbb{O}_{min}\not\subset [\g,x]$.
Note that $\mathbb{O}_{min}$ is $G$-stable. So if $\mathbb{O}_{min}\subset [\g,x]$,
then $[\g,x]$ contains a nonzero $G$-stable subspace. Since $\g$ is a simple Lie algebra,
this is impossible. This contradiction finishes the proof.
\end{proof}

\begin{Rem}\label{Rem:dim_preimage}
Note that a direct analog of this lemma holds for $M_n$: if $\pi$ denotes the projection
$(x,y,i,j)\mapsto x$, then $\dim \pi^{-1}(\mathbb{O})<n^2+2n-2$. The proof essentially repeats
that of Lemma \ref{Lem:dim_preimage}.
\end{Rem}

\subsection{Proof of Theorem \ref{Thm:X_n_properties}}
The proof requires two technical statements. Consider the regular semisimple locus $\g^{reg}\subset \g$
and set $X_n^{reg}:=\pi^{-1}(\g^{reg})$. This is an open subscheme in $X_n$.

\begin{enumerate}
\item We will show that $X_n^{reg}$ is dense in $X_n$.
\item We will show that $X_n^{reg}$ is irreducible.
\end{enumerate}

These two statements are proved in the lemmas below. After that we will easily
finish the proof of Theorem \ref{Thm:X_n_properties}.


\begin{Lem}\label{Lem:density}
$X_n^{reg}$ is dense in $X_n$.
\end{Lem}
\begin{proof}
The subscheme $X_n$ is defined by $\dim \g$ equations in a vector space of dimension
$2\dim\g+2n$. So the dimension of every irreducible component of $X_n$ is at least
$\dim\g+2n$. For the sake of contradiction, let $Z$ be a component that does not intersect $X_n^{reg}$. Let $p:=(x,y,0)$
be a point in a closed orbit in $Z$ that maps to a Zariski generic point in
the image of $Z$ in $\g\quo G$ (via $Z\xrightarrow{\pi}\g\rightarrow \g\quo G$).
Let $L$ be the centralizer of $x$ and $y$. Recall, Lemma
\ref{Lem:local_structure} that  locally near $Gp$ the scheme $X_n$ looks like
$\left(G\times^L (\C^{2k}\times\prod_{i=1}^k M_{n_i}\times X_{n_0})\right)^{\wedge_{G/L}}$.
The stabilizer of every closed orbit in $Z$ contains a conjugate of $L$ by the construction
of the latter.
It follows that the image of $\left(G\times^L (\C^{2k}\times\prod_{i=1}^k M_{n_i}\times X_{n_0})\right)^{\wedge_{G/L}}\cap Z$ under taking the categorical quotient
lies in $(\C^{2k})^{\wedge_0}$. So the intersection of $Z$ with
$(\prod_{i=1}^k M_{n_i}\times X_{n_0})^{\wedge_0}$ lies in the nilpotent locus.
Using Lemma \ref{Lem:dim_preimage} (for the $X_{n_0}$-factor) and
Remark \ref{Rem:dim_preimage} (for the $M_{n_i}$-factor) we conclude
that the dimension of the intersection does not exceed $\sum_{i=1}^k (n_i^2+2n_i-2-1)+
2n_0^2+3n_0-1$. Therefore the dimension of $Z$ does not exceed
$$\dim G/L+2k+\sum_{i=1}^k (n_i^2+2n_i-2-1)+2n_0^2+3n_0-1=\dim \g+2n-1.$$
This contradicts the observation that $\dim Z\geqslant \dim \g+2n$.
\end{proof}

\begin{Lem}\label{Lem:irreducibility}
The scheme $X_n^{reg}$ is irreducible.
\end{Lem}
\begin{proof}
Recall that $\g^{reg}=G\times^{N_G(T)}\mathfrak{h}^{reg}$. Set $X_n^0=\{(x,y,i)|
x\in \mathfrak{t}^{reg}, [x,y]=i^2\}$ so that $X_n^{reg}=G\times^{N_G(T)}X_n^0$.
We need to show that the Weyl group $W=N_G(T)/T$ acts transitively on the irreducible
components of $X_n^0$. Set $Y_n=\{i\in \C^{2n}| i^2\in \mathfrak{h}^\perp\}$.
We have a forgetful map $X_n^0\rightarrow Y_n\times \mathfrak{h}^{reg}$
forgetting $y$. This map is an affine bundle. So we need to show that $W$
transitively acts on the set of irreducible components of $Y_n$.
Let $p_1,\ldots,p_n,q_1,\ldots, q_n$ be Darboux coordinates on $\C^{2n}$.
Then $Y_n$ is given by
$$\{(p_1,\ldots,p_n,q_1,\ldots,q_n)| p_iq_i=0, \forall i=1,\ldots,n\}.$$
So there are $2^n$ irreducible components of $Y_n$: we need to choose if $p_i=0$
or $q_i=0$ for all $i=1,\ldots,n$. It is clear that $W$ permutes these components
transitively -- in fact, $\{\pm 1\}^n\subset W$ acts simply transitively on them.
\end{proof}

\begin{proof}[Proof of Theorem \ref{Thm:X_n_properties}]
Lemmas \ref{Lem:density} and \ref{Lem:irreducibility} imply that $X_n$ is irreducible.
Note that $\mu$ is a submersion at points with free $G$-orbit. There is a point in $X_n$
with free orbit: for example we can take $(x,0,i)$, where $x\in \mathfrak{h}^{reg}$
and $i$ is given by (in the notation of the proof of Lemma \ref{Lem:irreducibility})
by $p_1=\ldots=p_n=1,q_1=\ldots=q_n=0$. It follows that $\dim X_n=\dim \g+2n$
and that $X_n$ is generically reduced. Since $X_n$ is a complete intersection,
we see that $X_n$ is reduced.
\end{proof}

\section{Application to commuting scheme}
The goal of this section is to prove Theorem \ref{Thm:X_n_quotient}. Namely, we have inclusions
$\h^{\oplus 2}\hookrightarrow C_2(\g)\hookrightarrow X_n$ that give rise to
morphisms of categorical quotients
\begin{equation}\label{eq:quotient_morphisms}
\h^{\oplus 2}/W\rightarrow C_2(\g)\quo G\rightarrow X_n\quo G.
\end{equation}

\begin{proof}[Proof of Theorem \ref{Thm:X_n_quotient}]
We need to prove that the morphisms in (\ref{eq:quotient_morphisms}) are isomorphisms.
We note that the second morphism is a closed embedding. Also thanks to Corollary
\ref{Cor:closed_orbits} it is bijective. Thanks to Theorem \ref{Thm:X_n_properties},
$X_n$ is reduced, hence so is $X_n\quo G$. It follows that the second morphism in
(\ref{eq:quotient_morphisms}) is an isomorphism, and, in particular,
$C_2(\g)\quo G$ is reduced.

The first morphism in (\ref{eq:quotient_morphisms}) is bijective. It remains to show that
it is a full embedding, equivalently, that the pullback homomorphism
$\C[C_2(\g)]^G\rightarrow \C[\h^{\oplus 2}]^W$ is an isomorphism. Note that
both algebras are Poisson. For the source this holds because $\C[C_2(\g)]^G$
is obtained from $\C[\g^{\oplus 2}]$ by Hamiltonian reduction. The homomorphism
$\C[C_2(\g)]^G\rightarrow \C[\h^{\oplus 2}]^W$ intertwines the Poisson brackets.
To see this note that this homomorphism becomes a Poisson isomorphism after tensoring with
$\C[\g^{reg}]^G\cong \C[\h^{reg}]^W$ (in the first coordinate) and $\C[\h^{\oplus 2}]^W\hookrightarrow
\C[\h^{reg}]^W\otimes_{\C[\h]^W}\C[\h^{\oplus 2}]^W$. By results of \cite{Wallach},
the Poisson algebra $\C[\h^{\oplus 2}]^W$ is generated by the two subalgebras $\C[\h]^W$
(in the first and the second coordinates). To finish the proof it remains to notice that
the homomorphism $\C[C_2(\g)]^G\rightarrow \C[\h^{\oplus 2}]^W$ restricts to
$\C[\g]^G\xrightarrow{\sim} \C[\h]^W$ (for both the first and the second copy).
\end{proof}

\end{document}